\documentclass[11pt,a4paper]{amsart}
\usepackage{graphicx,multirow,array,amsmath,amssymb}

\newtheorem{theorem}{Theorem}

\newtheorem{proposition}[theorem]{Proposition}
\newtheorem{lemma}[theorem]{Lemma}

\begin{document}

\title{Bipartite intrinsically knotted graphs with 22 edges}
\author[H. Kim]{Hyoungjun Kim}
\address{Department of Mathematics, Korea University, Anam-dong, Sungbuk-ku, Seoul 136-701, Korea}
\email{kimhjun@korea.ac.kr}
\author[T. Mattman]{Thomas Mattman}
\address{Department of Mathematics and Statistics, California State University, Chico, Chico CA 95929-0525, USA}
\email{TMattman@CSUChico.edu}
\author[S. Oh]{Seungsang Oh}
\address{Department of Mathematics, Korea University, Anam-dong, Sungbuk-ku, Seoul 136-701, Korea}
\email{seungsang@korea.ac.kr}


\begin{abstract}
A graph is intrinsically knotted if every embedding contains a knotted cycle.
It is known that intrinsically knotted graphs have at least 21 edges 
and that the KS graphs, $K_7$ and the 13 graphs obtained from $K_7$ by $\nabla Y$ moves, are 
the only minor minimal intrinsically knotted graphs with 21 edges \cite{BM, JKM, LKLO, M}.
This set includes exactly one bipartite graph, the Heawood graph.

In this paper we classify the intrinsically knotted bipartite graphs with at most 22 edges. 
Previously known examples of intrinsically knotted graphs of size 22 were 
those with KS graph minor and the 168 graphs in the $K_{3,3,1,1}$ and $E_9+e$ families. 
Among these, the only bipartite example with no Heawood subgraph is Cousin 110 of  the $E_9+e$ family. 
We show that, in fact, this is a complete listing. 
That is, there are exactly two graphs of size at most 22 that are minor minimal bipartite intrinsically knotted: 
the Heawood graph and Cousin 110.
\end{abstract}

\maketitle

\section{Introduction} \label{sec:intro}

Throughout the paper, an embedded graph will mean one embedded in $R^3$.
A graph is {\em intrinsically knotted\/} if every embedding contains a non-trivially knotted cycle.
Conway and Gordon \cite{CG} showed that $K_7$, the complete graph with seven vertices,
is an intrinsically knotted graph.
Foisy \cite{F} showed that $K_{3,3,1,1}$ is also intrinsically knotted.
A graph $H$ is a {\em minor\/} of another graph $G$
if it can be obtained by contracting edges in a subgraph of $G$.
If a graph $G$ is intrinsically knotted and has no proper minor
that is intrinsically knotted, we say $G$ is  {\em minor minimal intrinsically knotted\/}.
Robertson and Seymour \cite{RS} proved that for any property of graphs, 
there is a finite set of graphs minor minimal with respect to that property. 
In particular,  there are only finitely many minor minimal intrinsically knotted graphs,
but finding the complete set is still an open problem.
A $\nabla Y$ {\em move\/} is an exchange operation on a graph
that removes all edges of a triangle $abc$ and
then adds a new vertex $v$ and three new edges $va, vb$ and $vc$.
The reverse operation is called a $Y \nabla$ {\em move\/} as follows:

\begin{figure}[h]
\includegraphics[scale=1]{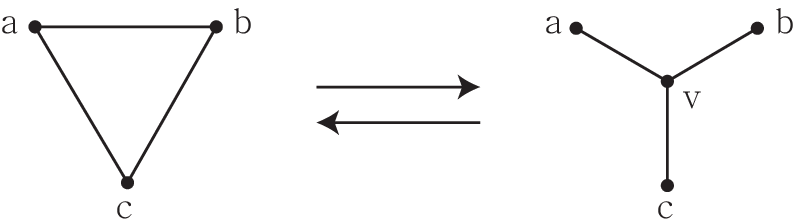}
\end{figure}

Since the $\nabla Y$ move preserves intrinsic knottedness \cite{MRS},
we will concentrate on triangle-free graphs.
We say two graphs $G$ and $G^{\prime}$ are {\em cousins} of each other
if $G^{\prime}$ is obtained from $G$ by a finite sequence of $\nabla Y$ and $Y \nabla$ moves.
The set of all cousins of $G$ is called the $G$ {\em family}.

Johnson, Kidwell and Michael \cite{JKM} and, independently, the second author \cite{M} showed 
that intrinsically knotted graphs have at least 21 edges.
Hanaki, Nikkuni, Taniyama and Yamazaki \cite{HNTY} constructed the $K_7$ family 
which consists of 20 graphs derived from $H_{12}$ and $C_{14}$ by $Y \nabla$ moves as in Figure \ref{fig11},
and they showed that the six graphs $N_9$, $N_{10}$, $N_{11}$, $N'_{10}$, $N'_{11}$ and $N'_{12}$
are not intrinsically knotted.
Goldberg, Mattman and Naimi \cite{GMN} also proved this, independently.

\begin{figure}[h]
\includegraphics[scale=1]{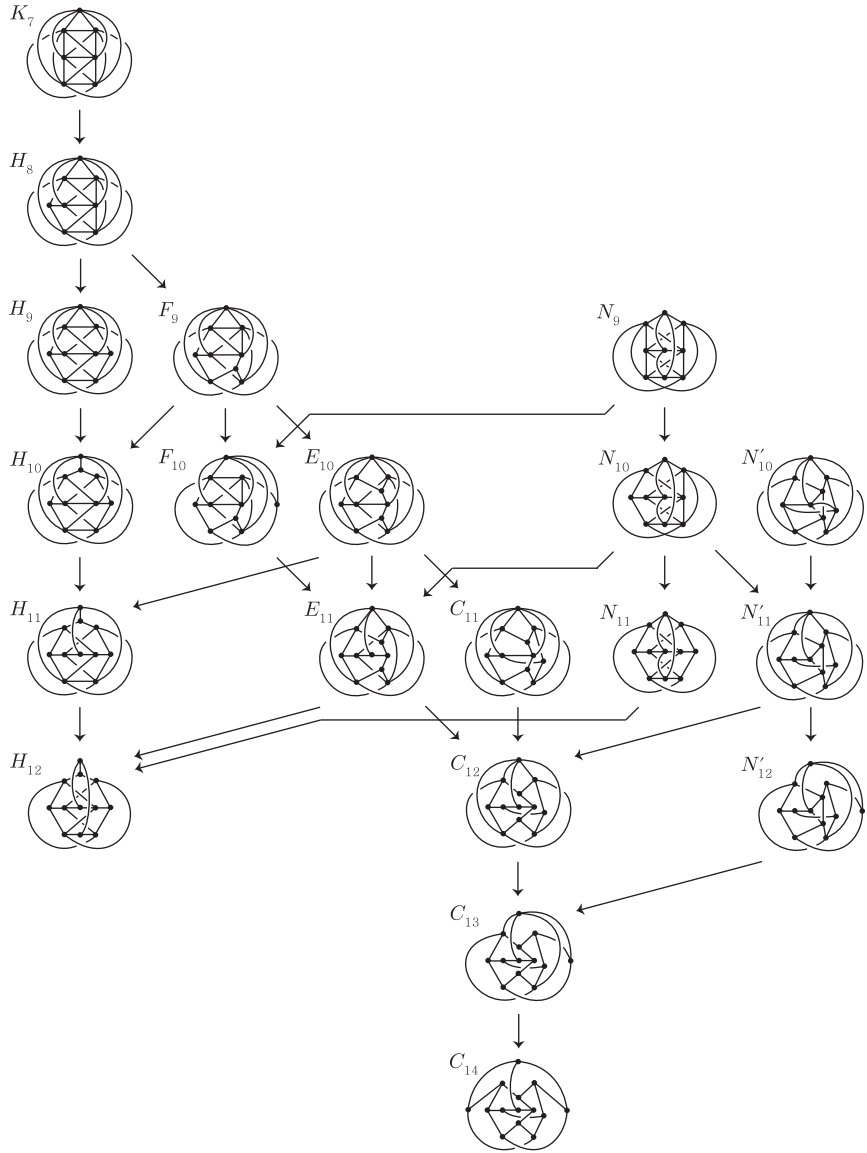}
\caption{The $K_7$ family}
\label{fig11}
\end{figure}

Recently two groups \cite{BM, LKLO}, working independently,
showed that $K_7$ and the 13 graphs obtained from $K_7$ by $\nabla Y$ moves are 
the only intrinsically knotted graphs with 21 edges. 
This gives us the complete set of 14 minor minimal intrinsically knotted graphs with 21 edges, 
which we call {\em the KS graphs} as they were first described by Kohara and Suzuki~\cite{KS}.

In this paper, we concentrate on intrinsically knotted graphs with 22 edges.
The $K_{3,3,1,1}$ family consists of 58 graphs,
of which 26 graphs were previously known to be minor minimal intrinsically knotted.
Goldberg et al.~\cite{GMN} showed that the remaining 32 graphs are also minor minimal intrinsically knotted.
The graph $E_9+e$ of Figure~\ref{fig12} is obtained from $N_9$ by adding a new edge $e$ 
and has a family of 110 graphs.
All of these graphs are intrinsically knotted 
and exactly 33 are minor minimal intrinsically knotted \cite{GMN}.
By combining the $K_{3,3,1,1}$ family and the $E_9+e$ family,
all 168 graphs were already known to be intrinsically knotted graphs with 22 edges.

\begin{figure}[h]
\includegraphics[scale=1]{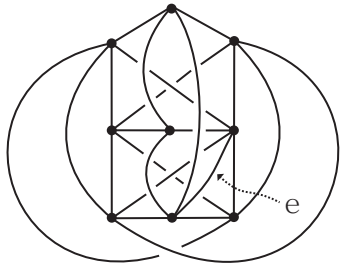}
\caption{The graph $E_9+e$}
\label{fig12}
\end{figure}

A {\em bipartite\/} graph is a graph whose vertices can be divided into two disjoint sets $A$ and $B$
such that every edge connects a vertex in $A$ to one in $B$.
Equivalently, a bipartite graph is a graph that does not contain any odd-length cycles.
Among the 14 intrinsically knotted graphs with 21 edges,
only $C_{14}$, {\em the Heawood graph}, is bipartite.

A bipartite graph formed by adding an edge to the Heawood graph will be bipartite intrinsically  knotted. 
We will show that this is the only way to form such a graph that has a KS graph minor.  
Among the remaining 168 known examples of intrinsically knotted graphs with 22 edges 
in the $K_{3,3,1,1}$ and $E_9+e$ families,
cousins 89 and 110 of the $E_9+e$ family are the only bipartite graphs. 

\begin{figure}[h]
\includegraphics[scale=1]{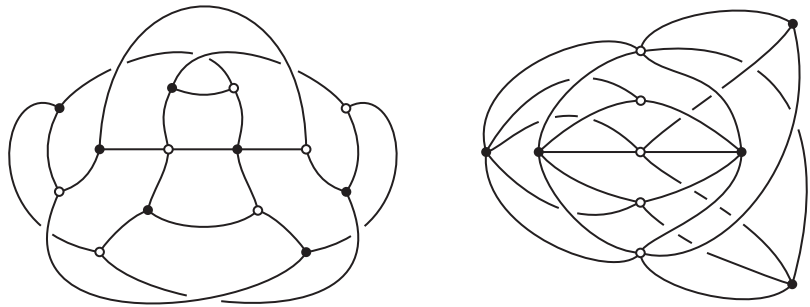}
\caption{Two cousins 89 and 110 of the $E_9+e$ family}
\label{fig13}
\end{figure}

However, Cousin 89 has the Heawood graph as a subgraph. 
Our goal in this paper is to show that Cousin 110 completes the list of minor minimal examples. 
We say that a graph $G$ is {\em minor minimal bipartite intrinsically knotted} 
if $G$ is an intrinsically knotted bipartite graph, but no proper minor of $G$ has this property. 
Since contracting edges can lead to a bipartite minor for a graph that was not bipartite to begin with, 
it's easy to construct examples of graphs that are not themselves bipartite intrinsically knotted 
even though they have a minor that is  minor minimal bipartite intrinsically knotted. 
Nonetheless, Robertson and Seymour's~\cite{RS} Graph Minor Theorem guarantees 
that there are a finite number of minor minimal bipartite intrinsically knotted graphs 
and every bipartite intrinsically knotted graph must have one as a minor. 
Our main theorem shows that there are exactly two such of 22 or fewer edges.

\begin{theorem}\label{thm:main}
There are exactly two graphs of size at most 22 that are minor minimal bipartite intrinsically knotted: 
The Heawood graph and Cousin 110 of the $E_9+e$ family.
\end{theorem}

As we show below, the argument quickly reduces to graphs of minimum degree $\delta(G)$ at least three, 
for which we have:

\begin{theorem}\label{thm:main3}
There are exactly two bipartite intrinsically knotted graphs with 22 edges and minimum degree at least three,
the two cousins 89 and 110 of the $E_9+e$ family.
\end{theorem}

We remark that Cousin 110 was earlier identified as bipartite intrinsically knotted in 
\cite[Theorem 3]{HAMM} as part of a classification of such graphs on ten or fewer vertices. 
It follows from that classification that Cousin 110 is 
the only minor minimal bipartite intrinsically knotted graph of order ten or less.
It would be interesting to know if there are further examples of order between 11 and 14, 
which is the order of the Heawood graph. 
Such examples would have at least 23 edges.

\begin{proof}[Proof of Theorem~\ref{thm:main}] 
Suppose $G$ is bipartite intrinsically knotted, with $\|G \| \leq 22$. 
If $\delta(G) \leq 1$, we may delete a vertex (and its edge, if it has one) to obtain a proper minor 
that also has this property, so $G$ is not minor minimal. 
If $\delta(G) = 2$, then contracting an edge adjacent to a degree two vertex gives a minor $H$ 
that remains intrinsically knotted and is of size at most 21. 
Thus $H$ is one of the KS graphs. 
In other words $G$ is obtained by a vertex split of the KS graph $H$. 
Now, a graph obtained in this way from a KS graph will be intrinsically knotted and have 22 edges. 
However, it's straightforward to verify that it cannot be bipartite.

So, we can assume $\delta(G) \geq 3$. 
If $\|G\| = 21$, $G$ must be a KS graph and
$C_{14}$, the Heawood graph, is the only bipartite graph in this set. 
As graphs of 20 edges are not intrinsically knotted~\cite{JKM,M}, 
$C_{14}$ is minor minimal for intrinsic knotting and, so, also for bipartite intrinsically knotted.

By Theorem~\ref{thm:main3}, if $\|G\| = 22$, $G$ must be one of the two cousins in the $E_9+e$ family. 
Goldberg et al.~\cite{GMN} showed that all graphs in this family are intrinsically knotted. 
However, Cousin 89 is formed by adding an edge to the Heawood graph and is not minor minimal. 
On the other hand, it's easy to verify that Cousin 110 is minor minimal bipartite intrinsically knotted and, 
therefore, the only such graph on 22 edges.
\end{proof}

The remainder of this paper is a proof of Theorem~\ref{thm:main3}. 
In the next section we introduce some terminology and outline the strategy of our proof.

\section{Terminology and strategy} \label{sec:term}

Henceforth, let $G=(A,B,E)$ denote a bipartite graph with 22 edges
whose partition has the parts $A$ and $B$ with $E$ denoting the edges of the graph.
Note that $G$ is triangle-free.
For any two distinct vertices $a$ and $b$,
let $\widehat{G}_{a,b}$ denote the graph
obtained from $G$ by deleting $a$ and $b$,
and then contracting edges adjacent to vertices of degree 1 or 2,
one by one repeatedly, until no vertices of degree 1 or 2 remain.
Removing vertices means deleting interiors of all edges adjacent to these vertices
and any remaining isolated vertices.
Let $\widehat{E}_{a,b}$ denote the set of edges of $\widehat{G}_{a,b}$.
The distance, denoted by ${\rm dist}(a,b)$, between $a$ and $b$ is the number of edges
in the shortest path connecting them.
If $a$ has the distance 1 from $b$, then we say that $a$ and $b$ are {\em adjacent\/}.
The degree of $a$ is denoted by $\deg(a)$.
Note that $\sum_{a \in A} \deg(a) = \sum_{b \in B} \deg(b) = 22$ by the definition of bipartition.
To count the number of edges of $\widehat{G}_{a,b}$, we have some notation.

\begin{itemize}
\item $E(a)$ is the set of edges that are adjacent to a vertex $a$.
\item $V(a)=\{c \in A \cup B\ |\ {\rm dist}(a,c)=1\}$
\item $V_n(a)=\{c \in A \cup B\ |\ {\rm dist}(a,c)=1,\ \deg(c)=n\}$
\item $V_n(a,b)=V_n(a) \cap V_n(b)$
\item $V_Y(a,b)=\{c \in A \cup B\ |\ \exists \ d \in V_3(a,b) \ \mbox{such that}
\ c \in V_3(d) \setminus \{a,b\}\}$
\end{itemize}

Obviously in $G \setminus \{a,b\}$ for some distinct vertices $a$ and $b$,
each vertex of $V_3(a,b)$ has degree 1.
Also each vertex of $V_3(a), V_3(b)$ (but not of $V_3(a,b)$) and $V_4(a,b)$ has degree 2.
Therefore to derive $\widehat{G}_{a,b}$ all edges adjacent to $a,b$ and $V_3(a,b)$
are deleted from $G$,
followed by contracting one of the remaining two edges adjacent to each vertex of
$V_3(a)$, $V_3(b)$, $V_4(a,b)$ and $V_Y(a,b)$ as in Figure \ref{fig21}(a).
Thus we have the following equation counting the number of edges of $\widehat{G}_{a,b}$
which is called the {\em count equation\/};
$$|\widehat{E}_{a,b}| = 22 - |E(a)\cup E(b)| -
(|V_3(a)|+|V_3(b)|-|V_3(a,b)|+|V_4(a,b)|+|V_Y(a,b)|).$$

\begin{figure}[h]
\includegraphics[scale=1]{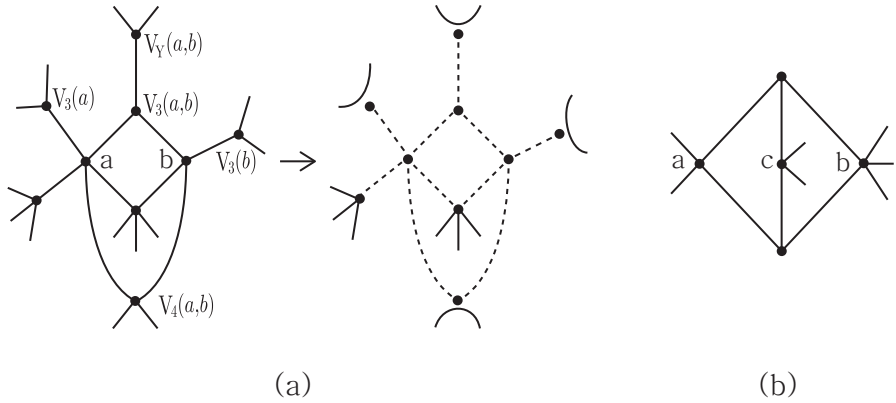}
\caption{Deriving $\widehat{G}_{a,b}$}
\label{fig21}
\end{figure}

For short, write $NE(a,b) = |E(a)\cup E(b)|$ and $NV_3(a,b) = |V_3(a)|+|V_3(b)|-|V_3(a,b)|$.
If $a$ and $b$ are adjacent (i.e. dist$(a,b)=1$),
then $V_3(a,b), V_4(a,b)$ and $V_Y(a,b)$ are all empty sets because $G$ is triangle-free.
Note that the derivation of $\widehat{G}_{a,b}$ must be handled slightly differently
when there is a vertex $c$ in $V$ such that more than one vertex of $V(c)$ is contained
in $V_3(a,b)$ as in Figure \ref{fig21}(b).
In this case we usually delete or contract more edges even though $c$ is not in $V_Y(a,b)$.

The following proposition, which was mentioned in \cite{LKLO},
gives two important conditions that ensure a graph fails to be intrinsically knotted.
Note that $K_{3,3}$ is a triangle-free graph and every vertex has degree 3.

\begin{proposition} \label{prop:planar}
If $\widehat{G}_{a,b}$ satisfies one of the following two conditions,
then $G$ is not intrinsically knotted.
\begin{itemize}
\item[(1)] $|\widehat{E}_{a,b}| \leq 8$, or
\item[(2)] $|\widehat{E}_{a,b}|=9$ and $\widehat{G}_{a,b}$ is not isomorphic to $K_{3,3}$.
\end{itemize}
\end{proposition}

\begin{proof}
If $|\widehat{E}_{a,b}| \leq 8$, then $\widehat{G}_{a,b}$ is a planar graph.
Also if $|\widehat{E}_{a,b}|=9$, then $\widehat{G}_{a,b}$ is
either a planar graph or isomorphic to $K_{3,3}$.
It is known that if $\widehat{G}_{a,b}$ is planar,
then $G$ is not intrinsically knotted \cite{{BBFFHL},{OT}}.
\end{proof}

In proving Theorem~\ref{thm:main3}
it is sufficient to  consider connected graphs having no vertex of degree 1 or 2.
Our process is to construct all possible such bipartite graphs $G$ with 22 edges,
delete two vertices $a$ and $b$ of $G$,
and then count the number of edges of $\widehat{G}_{a,b}$.
If $\widehat{G}_{a,b}$ has 9 edges or less and is not isomorphic to $K_{3,3}$,
then we conclude that $G$ is not intrinsically knotted by Proposition \ref{prop:planar}.

\begin{proposition} \label{prop:deg6}
There is no bipartite intrinsically knotted graph with 22 edges and minimum
degree at least three that has a vertex of degree 6 or more.
\end{proposition}

\begin{proof}
Suppose that $G$ is an intrinsically knotted graph with 22 edges
that has a vertex $a$ in $A$ of degree 6 or more.
Since $\sum_{b \in B} \deg(b) = 22$ and each vertex of $B$ has degree at least 3,
$B$ consists of at most seven vertices, so the degree of $a$ cannot exceed 7.

If $\deg(a) = 7$, then $B$ consists of seven vertices, 
and so one vertex $b$ has degree 4 and the others have degree 3.
Then $NE(a,b) = 10$ and $|V_3(a)| = 6$.
By the count equation, $|\widehat{E}_{a,b}| \leq 6$ in $\widehat{G}_{a,b}$.

Now assume that $\deg(a) = 6$. 
Since $\sum_{c \in A} \deg(c) = 22$, there is a vertex $c$ of degree at least 4 in $A$.
We may assume that $|V_3(a)| + |V_4(a,c)| \leq 3$, 
otherwise $|\widehat{E}_{a,c}| \leq 8$ because $NE(a,c) \geq 10$.
Because $|V_3(a)| \leq 3$, $B$ consists of exactly six vertices and 
at most three vertices in $B$ have degree 3. 
Furthermore $c$ is adjacent to at most three vertices of degree 3 or 4 in $B$.
Therefore, eventually $c$ has degree 4, 
and is adjacent to a degree 5 vertex $b$ and three other degree 3 vertices in $B$
as shown in Figure \ref{fig22}.
If there is another vertex of degree more than 3 in $A$,
then, like $c$, it is adjacent to three degree 3 vertices in $B$.
Therefore $A$ can have at most three vertices of degree more than 3, including $a$ and $c$.
This implies that $V_3(b) \geq 2$, and so $|\widehat{E}_{a,b}| \leq 7$.
\end{proof}

\begin{figure}[h]
\includegraphics[scale=1]{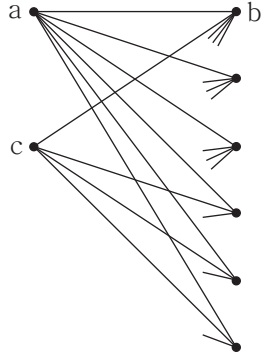}
\caption{The case of $\deg(a) = 6$}
\label{fig22}
\end{figure}

Therefore each vertex of $A$ and $B$ has degree 3, 4 or 5 only.
Let $A_n$ denote the set of vertices in $A$ of degree $n = 3,4,5$ 
and $[A] = [|A_5|, |A_4|,|A_3|]$ and similarly for $B$.
Then $[A] = [3,1,1]$, $[2,3,0]$, $[2,0,4]$, $[1,2,3]$, $[0,4,2]$ or $[0,1,6]$.
Without loss of generality, we may assume that $|A_5| \geq |B_5|$,
and if $|A_5| = |B_5|$ then $|A_4| \geq |B_4|$.

This paper relies on the technical machinery developed in \cite{LKLO}.
We divide the proof of Theorem \ref{thm:main3} according to the size of $A_5$, 
always under the assumption that our graphs are of minimum degree at least three.
In Section 3, we show that the only bipartite intrinsically knotted graph
with two or more degree 5 vertices in $A$ is Cousin 110 of the $E_9+e$ family.
In Section 4, we show that there is no bipartite intrinsically knotted graph
with exactly one degree 5 vertex in $A$.
In Section 5, we show that the only bipartite intrinsically knotted graph
with all vertices of degree at most 4 is Cousin 89 of the $E_9+e$ family.

\section{Case of $|A_5| \geq 2$} 

In this case, $A$ has at least two degree 5 vertices, say $a$ and $b$, 
and so $[A]$ is one of $[3,1,1]$, $[2,3,0]$ or $[2,0,4]$.

First we consider the case that $[B]$ is $[3,1,1]$, in other words, $([A], [B]) = ([3,1,1], [3,1,1])$.
The edges adjacent to each vertex in $A_5$ and $B_5$ are constructed in a unique way 
because both $A$ and $B$ have exactly five vertices.
This determines 21 edges of the graph
so that both $A$ and $B$ have three degree 5 vertices and two degree 3 vertices. 
Now we add a final, dashed edge to connect two degree 3 vertices, one in $A$ and another in $B$.
We get the graph shown in Figure \ref{fig31} (a)
which is Cousin 110 of the $E_9+e$ family and is intrinsically knotted.

In case $[B]$ is $[2,3,0]$, i.e., $([A], [B])$ is either $([3,1,1], [2,3,0])$ or $([2,3,0], [2,3,0])$,
we similarly construct the edges adjacent to each vertex in $A_5$ and $B_5$ in a unique way.
We add the remaining dashed edges which are also determined uniquely.
This gives, for the two cases, the graphs shown 
in Figure \ref{fig31} (b) and (c), respectively.
In both cases, $\widehat{G}_{a,b}$ is planar for the vertices $a$ and $b$ shown in the figures.

\begin{figure}[h]
\includegraphics[scale=1]{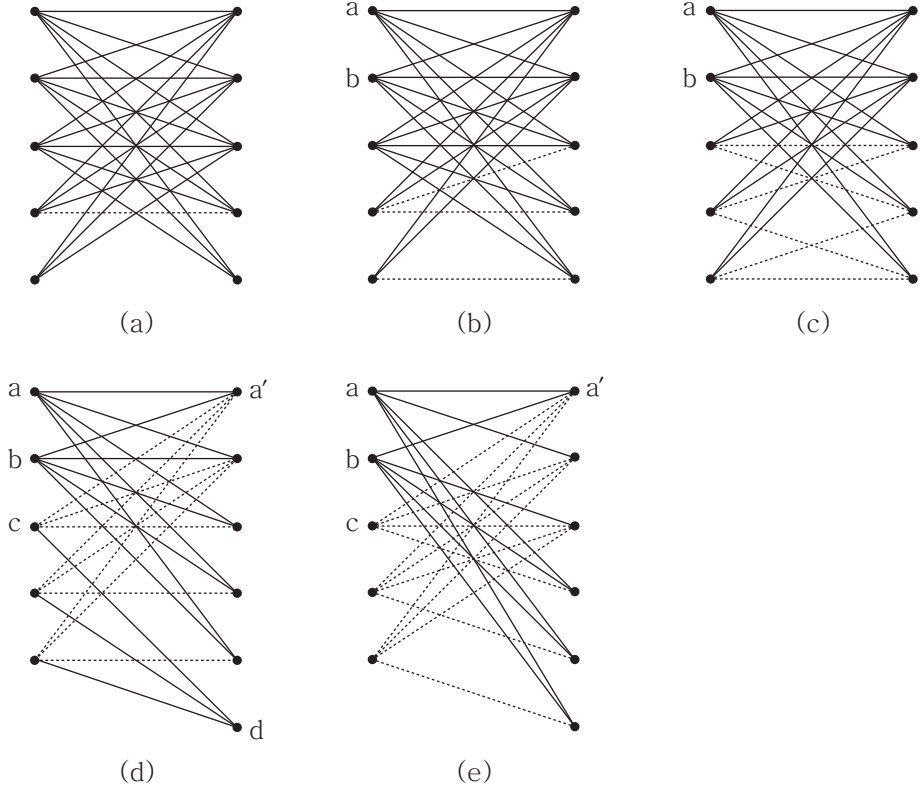}
\caption{The cases that $A$ has at least two degree 5 vertices}
\label{fig31}
\end{figure}

Now consider the case where $[B]$ is $[2,0,4]$, 
i.e., $([A], [B])$ is one of $([3,1,1], [2,0,4])$, $([2,3,0], [2,0,4])$ or $([2,0,4], [2,0,4])$.
We may assume that $NV_3(a,b) \leq 3$, otherwise $|\widehat{E}_{a,b}| \leq 8$ because $NE(a,b) \geq 10$.
Thus $a$ and $b$ are adjacent to the same five vertices in $B$ as shown in Figure \ref{fig31} (d).
Let $d$ be the remaining degree 3 vertex in $B$,
with $d$ adjacent to three vertices other than $a$ and $b$.
If $[A]$ is $[3,1,1]$, the remaining vertex in $A_5$ must be adjacent to the same vertices as $a$.
This is impossible because this vertex is also adjacent to $d$. 
If $[A]$ is $[2,3,0]$, the remaining dashed edges can be added in a unique way.
Let $c$ be a vertex in $A_4$.
Since $NE(a,c) = 9$ and $NV_3(a,c) = 4$,  then $|\widehat{E}_{a,c}| \leq 9$.
Since $\widehat{G}_{a,c}$ has the degree 4 vertex $b$, it is not isomorphic to $K_{3,3}$.
If $[A]$ is $[2,0,4]$, let $a'$ be a vertex in $B_5$,
then $NE(a,a') = 9$ and $NV_3(a,a') \geq 6$,  so $|\widehat{E}_{a,a'}| \leq 7$.

Consider the case where $[B]$ is $[1,2,3]$, 
i.e., $([A], [B])$ is one of $([3,1,1], [1,2,3])$, $([2,3,0], [1,2,3])$ or $([2,0,4], [1,2,3])$.
Similarly we may assume that $NV_3(a,b) + |V_4(a,b)| \leq 3$, otherwise $|\widehat{E}_{a,b}| \leq 8$.
Thus $a$ and $b$ are adjacent to the same four vertices of degree 5 or 3, 
but different degree 4 vertices in $B$ as shown in Figure \ref{fig31} (e).
If $[A]$ is $[3,1,1]$, the remaining vertex in $A_5$ must be adjacent to another 
degree 4 vertex
which is not adjacent to $a$ and $b$, but this is impossible. 
For the remaining two cases, we follow the same argument as the last two cases when $[B]$ was $[2,0,4]$.

It remains to consider  $[B] = [0,4,2]$ or $[0,1,6]$.
For any two vertices $a$ and $b$ in $A_5$,
$NE(a,b) = 10$ and $NV_3(a,b) + |V_4(a,b)| \geq 4$,  and so $|\widehat{E}_{a,b}| \leq 8$.

\section{Case of $|A_5| = 1$} 

In this case, there is only one choice for $[A]$,  $[1,2,3]$.
Let $a$, $b_1$, $b_2$, $c_1$, $c_2$ and $c_3$ be the degree 5 vertex, 
the two degree 4 vertices and the three degree 3 vertices in $A$.
There are three choices for $[B]$, $[1,2,3]$, $[0,4,2]$ or $[0,1,6]$.
We divide into three subsections, one for each case.

\subsection{$[B] = [1,2,3]$}   \hspace{1cm}

Let $a'$, $b'_1$, $b'_2$, $c'_1$, $c'_2$ and $c'_3$
be the degree 5 vertex, the two degree 4 vertices and the three degree 3 vertices in $B$.
If $NV_3(a,a') \geq 5$, then $|\widehat{E}_{a,a'}| \leq 8$.
Therefore $NV_3(a,a') = 4$ and we get the subgraph shown in Figure \ref{fig41} (a).
Assume that $c_3$ and $c'_3$ are the remaining unused vertices.

\begin{figure}[h]
\includegraphics[scale=1]{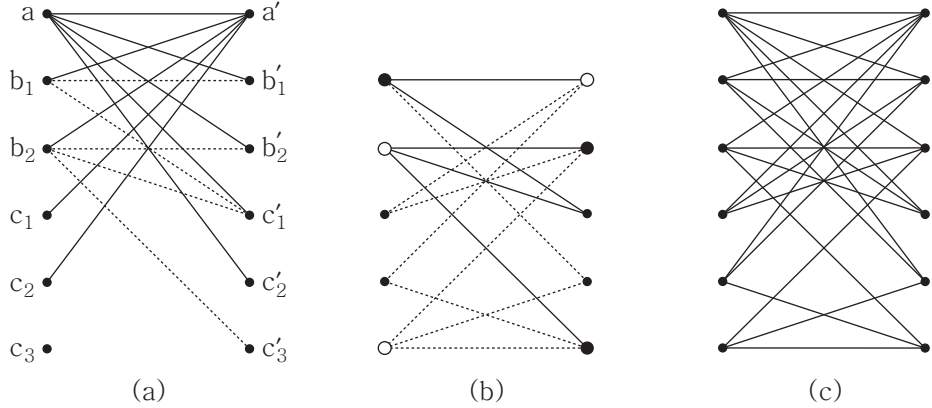}
\caption{$[A] = [B] = [1,2,3]$}
\label{fig41}
\end{figure}

If $V_4(b_1)$ is empty, i.e., $|V_3(b_1)| = 3$, then $|\widehat{E}_{b_1,a'}| \leq 9$ and 
$\widehat{G}_{b_1,a'}$ has the degree 4 vertex $a$.
Thus $|V_4(b_1)| \geq 1$, and similarly $|V_4(b_2)|$, $|V_4(b'_1)|$ and $|V_4(b'_2)| \geq 1$.
Without loss of generality, we say that ${\rm dist}(b_1,b'_1) =1$ and ${\rm dist}(b_2,b'_2) =1$.
If $V_4(c'_3)$ is empty, i.e., $|V_3(c'_3)| = 3$, then $|\widehat{E}_{a,c'_3}| \leq 9$ and 
$\widehat{G}_{a,c'_3}$ has the degree 4 vertex $b_1$.
Thus we may assume that ${\rm dist}(b_2,c'_3) =1$.
If ${\rm dist}(b_2,b'_1) =1$, then $|\widehat{E}_{a,b_2}| \leq 8$.
Thus we also assume that ${\rm dist}(b_2,c'_1) =1$.
If ${\rm dist}(c_i,c'_1) =1$ for some $i=1,2,3$, then $|\widehat{E}_{a,b_2}| \leq 8$
because $NV_3(a,b_2) + |V_4(a,b_2)| = 4$ and $V_Y(a,b_2) = \{ c_i \}$.
Therefore ${\rm dist}(b_1,c'_1) =1$.

Now consider the graph $\widehat{G}_{a,a'}$.
Since $|\widehat{E}_{a,a'}| \leq 9$, we only need to consider the case that 
$\widehat{G}_{a,a'}$ is isomorphic to $K_{3,3}$.
Since the three vertices $b_1$, $b'_2$ and $c'_3$ (big black dots in the figure)
are adjacent to $b_2$ in $\widehat{G}_{a,a'}$, 
they're also adjacent to $b'_1$ and $c_3$ (big white dots) as shown in Figure \ref{fig41} (b).
Restore the graph $G$ by adding back the two vertices $a$ and $a'$ and their associated nine edges 
as shown in Figure \ref{fig41} (c).
The reader can easily check that the graph $\widehat{G}_{a,b_1}$ is planar.

\subsection{$[B] = [0,4,2]$}   \hspace{1cm}

First, we give a very useful lemma.
Let $\widetilde{K}_{3,3}$ be the bipartite graph shown in Figure \ref{fig42}.
The six degree 3 vertices in $\widetilde{K}_{3,3}$ are divided into three big black vertices
and three big white vertices.
If we ignore the three degree 2 vertices, then we get $K_{3,3}$.
The vertex $d_4$ is called the {\em $s$-vertex\/}.

\begin{figure}[h]
\includegraphics[scale=1]{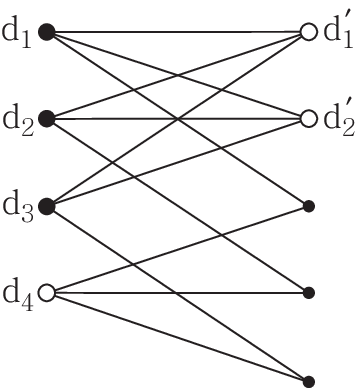}
\caption{$\widetilde{K}_{3,3}$}
\label{fig42}
\end{figure}

\begin{lemma}\label{lem:H}
Let $H$ be a bipartite graph such that one partition of its vertices contains four degree 3 vertices, 
and the other partition contains two degree 3 vertices and three degree 2 vertices.
If $H$ is not planar then $H$ is isomorphic to $\widetilde{K}_{3,3}$.
\end{lemma}

\begin{proof}
Let $d_1$, $d_2$, $d_3$ and $d_4$ be the  four degree 3 vertices in one partition of $H$,
and $d'_1$ and $d'_2$ be the two degree 3 vertices in the other partition, which contains degree 2 vertices.
Let $\widehat{H}$ be the graph obtained from $H$ by contracting three edges,
one each from the pair adjacent to the three degree 2 vertices.
Since $\widehat{H}$ consists of nine edges but is not planar,
it must be $K_{3,3}$.
Therefore $d'_1$, $d'_2$ and one of the $d_i$'s, say $d_4$,
are in the same partition of $K_{3,3}$ since ${\rm dist}(d'_1,d'_2) \geq 2$.
Since $H$ is originally a bipartite graph, 
$d_4$ is connected to $d_i$ for each $i=1,2,3$, by exactly two edges adjacent to each degree 2 vertex of $H$.
This gives the graph $\widetilde{K}_{3,3}$ shown in Figure \ref{fig42}.
\end{proof}

Let $b'_1$, $b'_2$, $b'_3$, $b'_4$, $c'_1$ and $c'_2$
be the four degree 4 vertices and the two degree 3 vertices in $B$.
First consider the case that $V_3(a)$ has only one vertex, say $c'_1$.
Note that ${\rm dist}(b_i,c'_1) = 1$ for each $i=1,2$, otherwise $|\widehat{E}_{a,b_i}| \leq 8$.
Now we divide into two cases depending on whether $b_1$ is adjacent to three vertices 
among the $b'_j$'s (say $b'_1$, $b'_2$ and $b'_3$) or two (say $b'_1$ and $b'_2$) along with $c'_2$.
In Figure \ref{fig43}, the ten non-dashed edges in figures (a) and (b) indicate the first case
while the ten non-dashed edges in figures (c)$\sim$(e) indicate the second.

Let $H$ be the bipartite graph $H$ obtained from $G$ by deleting these ten non-dashed edges.
Then $H$ has four degree 3 vertices from $A$, 
and two degree 3 vertices and three degree 2 vertices from $B$.
We only need to handle the case that $H$ is not planar
because if $H$ is planar, then $\widehat{G}_{a,b_1}$ is also planar.
By Lemma \ref{lem:H}, $H$ is isomorphic to $\widetilde{K}_{3,3}$.
In each case, the $s$-vertex is either at $b_2$ as in figures (a) and (c) or 
at one of $c_i$'s, say $c_1$, as in figures (b), (d) and (e). 
The big white vertex on the left identifies the $s$-vertex.
Indeed these five figures (a)$\sim$(e) represent all the possibilities up to symmetry.

\begin{figure}[h]
\includegraphics[scale=1]{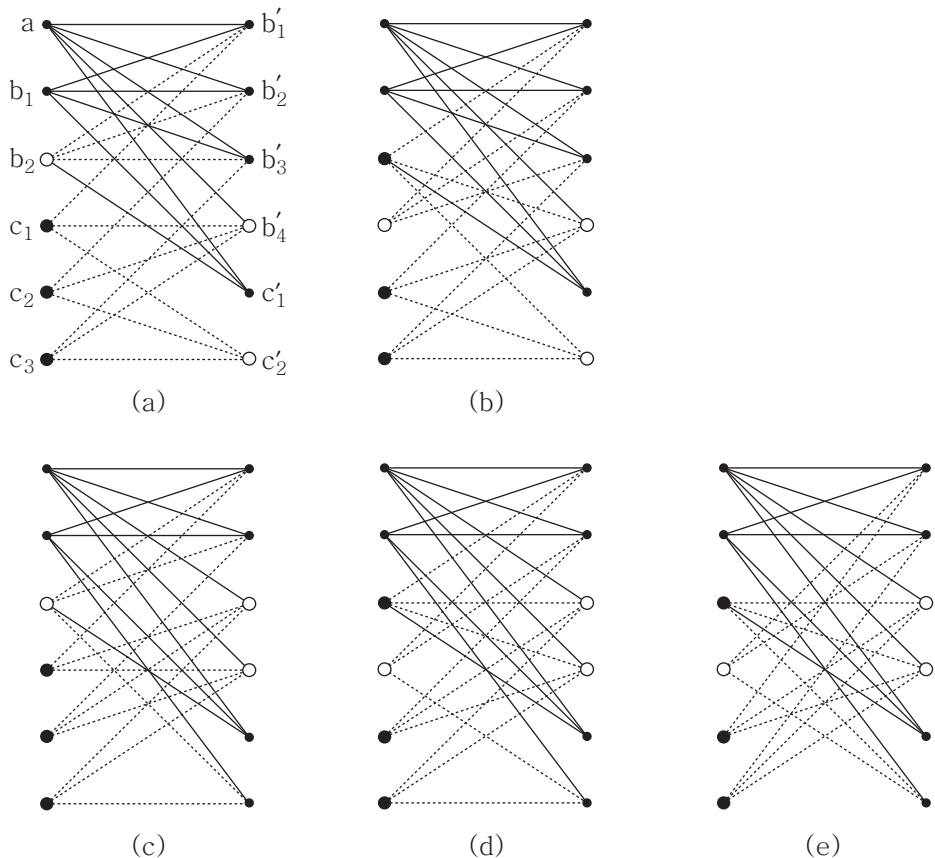}
\caption{$[A] = [1,2,3]$ and $[B] = [0,4,2]$ with $|V_3(a)|=1$}
\label{fig43}
\end{figure}

For the two graphs in the figures (a) and (c),
each $\widehat{G}_{b'_1,b'_2}$ has ten edges, but also contains two bi-gons and so is planar.
For the other three graphs in the figure, (b), (d) and (e),
each $\widehat{G}_{a,b_2}$ has nine edges, but also contains a bi-gon 
on the vertices $c_2$ and $c_3$. So, these are also planar.

Now consider the case where $V_3(a)$ has two vertices.
Thus $a$ is adjacent to three $b'_j$ vertices, say $b'_1$, $b'_2$ and $b'_3$,
as well as $c'_1$ and $c'_2$.
If $|V_3(b'_4)|=3$, then $NV_3(a,b'_4) = 5$, so $|\widehat{E}_{a,c}| \leq 8$.
We may assume that $b'_4$ is adjacent to $b_1$, $b_2$, $c_1$ and $c_2$ 
as shown in Figure \ref{fig44} (a).
Furthermore, $|V_3(b'_j)| \leq 2$ for each $j=1,2,3$, otherwise $|\widehat{E}_{a,b'_j}| \leq 9$ and 
$\widehat{G}_{a,b'_j}$ has the degree 4 vertex $b'_4$.
Thus each $b'_j$ is adjacent to either $b_1$ or $b_2$ or both.
Without loss of generality, we say that $b_1$ is adjacent to both $b'_1$ and $b'_2$.
Also $V_3(b_1)$ must have one vertex, say $c'_1$, and $c'_1$ is adjacent to $b_2$, 
otherwise $NV_3(a,b_1) + |V_4(a,b_1)| + |V_Y(a,b_1)| \geq 5$, so $|\widehat{E}_{a,b_1}| \leq 8$.
The four dashed edges in the figure indicate these new edges.

\begin{figure}[h]
\includegraphics[scale=1]{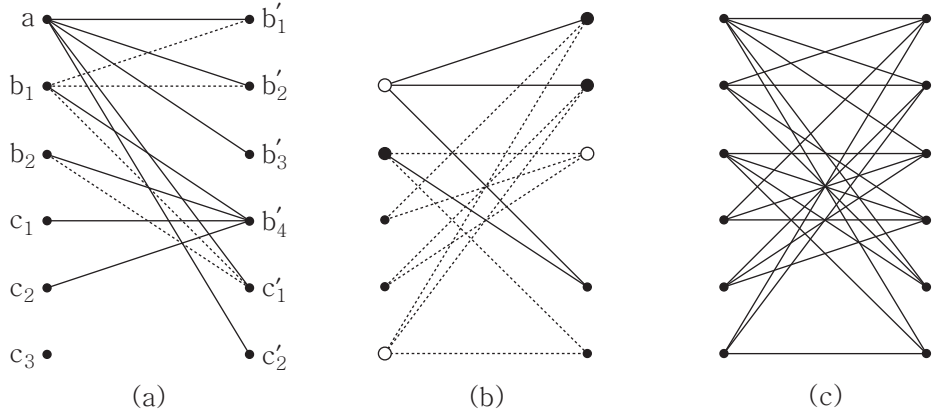}
\caption{$[A] = [1,2,3]$ and $[B] = [0,4,2]$ with $|V_3(a)|=2$}
\label{fig44}
\end{figure}

Now consider the graph $\widehat{G}_{a,b'_4}$.
Since $|\widehat{E}_{a,b'_4}| \leq 9$, 
we can assume that $\widehat{G}_{a,b'_4}$ is isomorphic to $K_{3,3}$.
Since the three vertices $b_2$, $b'_1$ and $b'_2$ (big black dots in the figure)
are adjacent to $b_1$ in $\widehat{G}_{a,b'_4}$, 
they're also adjacent to $c_3$ and $b'_3$ (big white dots)
as shown in Figure \ref{fig44} (b).
Restore the graph $G$ by adding back the two vertices $a$ and $b'_4$ and their associated nine edges
as shown in Figure \ref{fig44} (c).
The reader can easily check that the graph $\widehat{G}_{a,b_2}$ is planar.

\subsection{$[B] = [0,1,6]$}   \hspace{1cm}

Let $b'$ be the degree 4 vertex in $B$.
For given $i=1,2$, $NV_3(a,b_i) + |V_4(a,b_i)| \leq 4$, otherwise $|\widehat{E}_{a,b_i}| \leq 8$.
Therefore $V_3(a)$, $V_3(b_1)$ and $V_3(b_2)$ are the same set of four degree 3 vertices in $B$.
Since $|V_3(b')|=3$, then $|\widehat{E}_{a,b'}| = 7$.

\section{Case of $|A_5| = 0$}

In this case, $([A], [B])$ is one of $([0,4,2], [0,4,2])$, $([0,4,2], [0,1,6])$ or $([0,1,6], [0,1,6])$.
We divide into three subsections, one for each case.

\subsection{$([A], [B]) = ([0,4,2], [0,4,2])$}   \hspace{1cm}

We first give a very useful lemma.
Let $\widetilde{P}_{10}$ be the bipartite graph shown in Figure \ref{fig51}.
By deleting the vertex $c'$ and its two edges, we get $\widetilde{K}_{3,3}$.
The vertices $d_4$ and $c'$ are called the {\em $s$-} and {\em $t$-vertex\/}, respectively.

\begin{figure}[h]
\includegraphics[scale=1]{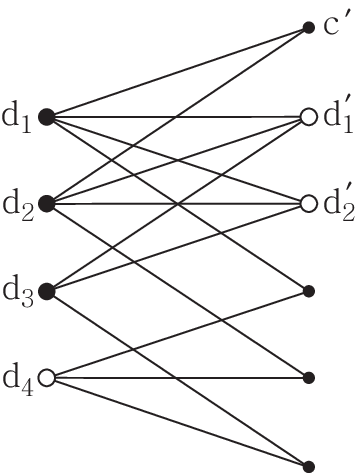}
\caption{$\widetilde{P}_{10}$}
\label{fig51}
\end{figure}

\begin{lemma}\label{lem:P}
Let $H$ be a bipartite graph such that one partition of its vertices contains two degree 4 vertices  
and two degree 3 vertices, and the other partition contains two degree 3 vertices and four degree 2 vertices.
If $H$ is not planar then $H$ is isomorphic to $\widetilde{P}_{10}$.
\end{lemma}

\begin{proof}
Let $d_1$, $d_2$, $d_3$ and $d_4$ be the two degree 4 vertices and the two degree 3 vertices
in one partition of $H$,
and $d'_1$ and $d'_2$ be the two degree 3 vertices in the other partition, which contains degree 2 vertices.
Let $\widehat{H}$ be the graph obtained from $H$ by contracting four edges,
one each from the pair adjacent to each of the degree 2 vertices.
Since $\widehat{H}$ consists of six vertices and ten edges but is not planar,
it must be $K_{3,3}+e$, the graph obtained from $K_{3,3}$ by connecting two vertices in the same partition
by an edge $e$.
Then $e$ must connect the two degree 4 vertices $d_1$ and $d_2$.
Furthermore $d_1$ and $d_2$ and one of $d_3$ or $d_4$, say $d_3$, are 
in the same partition of $K_{3,3}+e$ containing the edge $e$.
Since $H$ was originally a bipartite graph, 
$d_1$ and $d_2$ are connected by two edges adjacent to a degree 2 vertex, call it $c'$, of $H$,
and $d_4$ is connected to $d_i$ for each $i=1,2,3$, by two edges adjacent to each other degree 2 vertex of $H$.
So, we get the graph $\widetilde{P}_{10}$ as shown in Figure \ref{fig51}.
\end{proof}

First consider the case that some degree 4 vertex, say $b_1$, in $A$ 
is adjacent to all four degree 4 vertices in $B$.
Let $b_2$ denote another degree 4 vertex in $A$,
and $H$ the graph obtained from $G$ by deleting the two vertices $b_1$ and $b_2$ and the adjacent eight edges.
If the graph $\widehat{G}_{b_1,b_2}$ is not planar, then $H$ satisfies all assumptions of Lemma \ref{lem:P}.
Thus $H$ is isomorphic to $\widetilde{P}_{10}$.
Now restore the vertex $b_2$ and the associated four dashed edges as shown in Figure \ref{fig52} (a).
Note that these four edges can be replaced in a unique way because of the assumptions for the vertex $b_1$.
Let $b_3$ denote a degree 4 vertex in $A$, other than $b_1$ and $b_2$.
The reader can easily check that the graph $\widehat{G}_{b_1,b_3}$ is planar as shown in Figure \ref{fig52} (b).

\begin{figure}[h]
\includegraphics[scale=1]{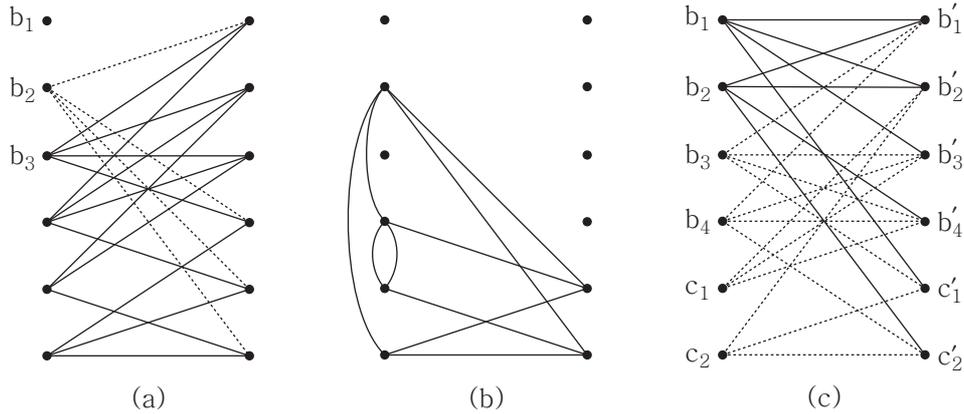}
\caption{$[A]=[B] = [0,4,2]$ with $|V_3(b_1)|=0$ or 1}
\label{fig52}
\end{figure}

Now assume that each degree 4 vertex, say $b_i$ for $i=1,2,3,4$, in $A$ is adjacent to 
at least one of the two degree 3 vertices, say $c'_1$ and $c'_2$, in $B$.
By counting the degrees of each vertex, we may assume that 
$b_1$ is adjacent to $c'_1$, but not $c'_2$.
Also assume that $b_2$ is adjacent to $c'_2$, but not $c'_1$
since at most three vertices among the $b_i$'s can be adjacent to $c'_1$.
If $V_4(b_1) = V_4(b_2)$, then $|\widehat{E}_{b_1,b_2}| \leq 9$.
Since $\widehat{G}_{b_1,b_2}$ has the remaining degree 4 vertex in $B$ outside of $V_4(b_1)$,
it is not isomorphic to $K_{3,3}$.
Therefore  $V_4(b_1) \cap V_4(b_2)$ has two vertices, say $b'_1$ and $b'_2$, in $B$.
As drawn in Figure \ref{fig52} (c), 
the eight non-dashed edges adjacent to the vertices $b_1$ and $b_2$ are settled.
Let $H$ denote the graph obtained from $G$ by deleting these two vertices and the associated eight edges.
If the graph $\widehat{G}_{b_1,b_2}$ is not planar, then $H$ satisfies all assumptions of Lemma \ref{lem:P}.
Thus $H$ is isomorphic to $\widetilde{P}_{10}$.
There are several choices for the $s$-vertex and the $t$-vertex among $A$ and $B$, respectively.
For example, we can choose $c_2$ for the $s$-vertex and $b'_1$ for the $t$-vertex
as in figure (c).
Whatever choice is made, the graph $\widehat{G}_{b_1,b_3}$ is always planar.

\subsection{$([A], [B]) = ([0,4,2], [0,1,6])$}   \hspace{1cm}

First assume that the degree 4 vertex, say $b'$, in $B$ is adjacent to
all four degree 4 vertices, say $b_1$, $b_2$, $b_3$ and $b_4$, in $A$.
Let $c'_1$ through $c'_6$ denote the six degree 3 vertices in $B$.
If $|V_3(b_j,b_k)| \geq 5$ for some different $j$ and $k$, then $|\widehat{E}_{b_j,b_k}| \leq 8$.
Thus we may assume that $|V_3(b_j,b_k)| \leq 4$ for all pairs $j$ and $k$,
so $V_3(b_j)$ and $V_3(b_k)$ have at least two common vertices among the $c'_i$'s.
Assume that $b_1$ is adjacent to $c'_1$, $c'_2$ and $c'_3$.
Note that each $c'_i$ is adjacent to at least one of the $b_j$'s because $c'_i$ has degree 3.
We also assume that $c'_4$ is adjacent to $b_2$.
Since $|V_3(b_1,b_2)| \leq 4$, assume that $b_2$ is also adjacent to $c'_1$ and $c'_2$.
Again, assume that  $c'_5$ is adjacent to $b_3$.
Since $|V_3(b_j,b_3)| \leq 4$ for $j=1,2$, $b_3$ must be adjacent to $c'_1$ and $c'_2$.
Now $c'_6$ is adjacent to $b_4$, and similarly $b_4$ must be adjacent to $c'_1$ and $c'_2$.
This is impossible because $c'_1$ and $c'_2$ have degree 3.
See Figure \ref{fig53} (a).

\begin{figure}[h]
\includegraphics[scale=1]{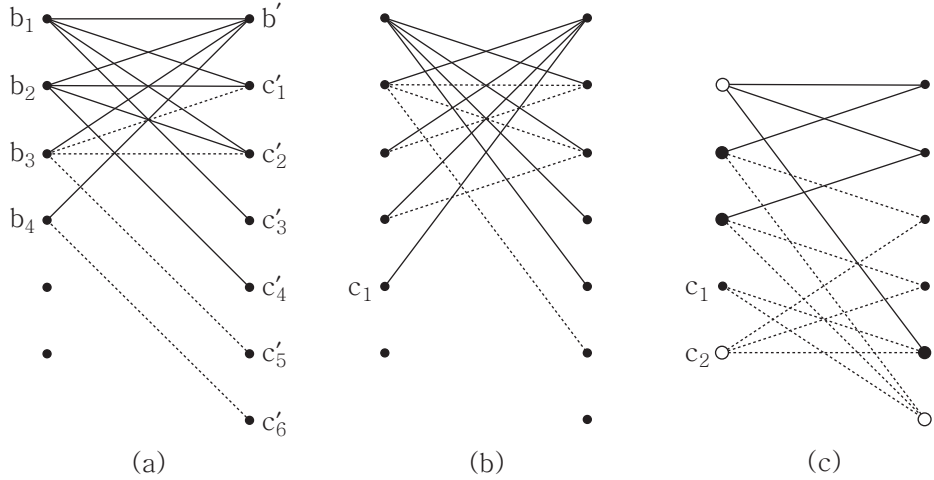}
\caption{$[A]= [0,4,2]$ and $[B] = [0,1,6]$}
\label{fig53}
\end{figure}

Now consider the case that $b'$ is not adjacent to a degree 4 vertex, say $b_1$, in $A$.
Assume that $b_1$ is adjacent to $c'_1$, $c'_2$, $c'_3$ and $c'_4$.
If $|V_3(b')| = 2$, then $|\widehat{E}_{b_1,b'}| \leq 8$.
Thus we may assume that $b'$ is adjacent to $b_2$, $b_3$, $b_4$ and a degree 3 vertex, say $c_1$.
We also assume that $c'_5$ is adjacent to $b_2$ because $c'_5$ has degree 3.
If $b_2$ is adjacent to $c'_6$ or $V_Y(b_1,b_2)$ is not empty, then $|\widehat{E}_{b_1,b_2}| \leq 8$.
Thus we may assume that $b_2$ is adjacent to $c'_1$ and $c'_2$,
and $c'_1$ and $c'_2$ are adjacent to $b_3$ and $b_4$, respectively. These are 
the non-dashed edges in Figure \ref{fig53} (b).

Now consider the graph $\widehat{G}_{b_1,b'}$.
Since $|\widehat{E}_{b_1,b'}| \leq 9$, we can assume that 
$\widehat{G}_{b_1,b'}$ is isomorphic to $K_{3,3}$.
Since the three vertices $b_3$, $b_4$ and $c'_5$ (big black dots in the figure)
are adjacent to $b_2$ in $\widehat{G}_{b_1,b'}$, 
they're also adjacent to $c_2$ and $c'_6$ (big white dots)
as shown in Figure \ref{fig53} (c).
Restore the graph $G$ by adding back the two vertices $b_1$ and $b'$ and their associated nine edges.
Then $NV_3(b_2,b_3) + |V_4(b_2,b_3)| = 6$, so $|\widehat{E}_{b_2,b_3}| \leq 8$.

\subsection{$([A], [B]) = ([0,1,6], [0,1,6])$}   \hspace{1cm}

If the degree 4 vertex, say $b$, in $A$ is not adjacent to the degree 4 vertex, say $b'$, in $B$,
then $|\widehat{E}_{b,b'}| \leq 6$ because $NV_3(b,b') = 8$.
Assume that $b$ is adjacent to $b'$, and let $e$ denote the edge connecting these two vertices.

\begin{lemma}\label{4cycle}
In this case, if $G$ is intrinsically knotted, then every 4-cycle contains the edge $e$.
\end{lemma}

\begin{proof}
Suppose that there is a 4-cycle $H$ which does not contain $e$. 
Then, we may assume that $H$ does not contain $b$.
Let $c$ be any vertex in $A$ such that 
either $b$ or $c$ is adjacent to some vertex of $H$ in $B$, other than $b'$.
Since $|V_3(b,c)| = |V_Y(b,c)|$ and $|V_3(c)| + |V_4(b,c)| = 3$,
$NV_3(b,c) + |V_4(b,c)| + |V_Y(b,c)| = (|V_3(b)| + |V_3(c)| - |V_3(b,c)|) + |V_4(b,c)| + |V_3(b,c)| = 6$.
Therefore $|\widehat{E}_{b,c}| \leq 9$.
Furthermore $\widehat{G}_{b,c}$ contains $H$ which is no longer a  4-cycle
because at least one vertex of $H$ in $B$ has degree 2.
So $\widehat{G}_{b,c}$ is not isomorphic to $K_{3,3}$.
\end{proof}

Now consider the subgraph $G \setminus \{e\}$ 
which consists of fourteen degree 3 vertices and has no 4-cycle by Lemma \ref{4cycle}.
We name these vertices $c_i$'s and $c'_j$'s as in Figure \ref{fig54} (a).
Assume that $c_1$ is adjacent to $c'_1$, $c'_2$ and $c'_3$, and
$c'_1$ is adjacent to $c_2$ and $c_3$.
Since there is no 4-cycle, we can also assume that
$c_2$ is adjacent to $c'_4$ and $c'_5$, $c_3$ with $c'_6$ and $c'_7$,
$c'_2$ with $c_4$ and $c_5$, and $c'_3$ with $c_6$ and $c_7$ as illustrated by the non-dashed edges in the figure.
Without loss of generality, we may assume that $c_4$ is adjacent to $c'_4$ and $c'_6$, 
and then $c_5$ must be adjacent to $c'_5$ and $c'_7$.
Similarly we may assume that $c_6$ is adjacent to $c'_4$, 
and then $c_6$ must be adjacent to $c'_7$, and $c_7$ to $c'_5$ and $c'_6$.
Finally we get the Heawood graph $C_{14}$ as drawn in Figure \ref{fig54} (b).
Note that $C_{14}$ is symmetric and any pair of vertices $c_i$ and $c'_j$ has  distance either 1 or 3.
Thus $G$ can be obtained by connecting two such vertices of distance 3 by the edge $e$.
This graph is isomorphic to Cousin 89 of the $E_9+e$ family as drawn in Figure \ref{fig13}.

\begin{figure}[h]
\includegraphics[scale=1]{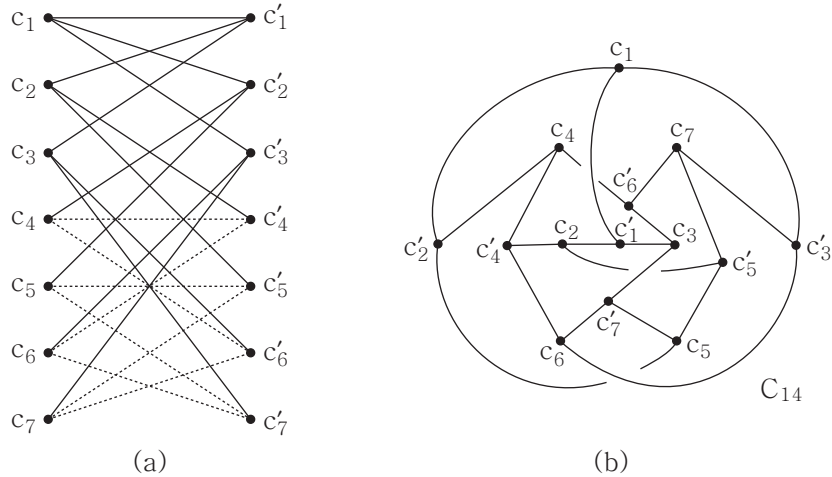}
\caption{The subgraph $G \setminus \{e\}$}
\label{fig54}
\end{figure}

\end{document}